\documentclass[11pt]{amsart}
\usepackage{amsmath,amsthm, amsfonts, amssymb, setspace, color, enumerate, breqn}
\usepackage[usenames,dvipsnames]{xcolor}
\definecolor{cobalt}{rgb}{0.0, 0.28, 0.67}
\definecolor{darkblue}{rgb}{0.0, 0.0, 0.55}
\usepackage[colorlinks,linkcolor=BrickRed,citecolor=darkblue,urlcolor=black,hypertexnames=true]{hyperref}

\usepackage{enumerate}

      \newtheorem{theorem}{Theorem}[section]
      \newtheorem{example}[theorem]{Example}
      
      \newtheorem{remark}[theorem]{Remark}
       \newtheorem{corollary}[theorem]{Corollary}
      \newtheorem{lemma}[theorem]{Lemma}
      
      \newtheorem{proposition}[theorem]{Proposition}
      
      \def\R{{\mathbb R}}
      \def\N{{\mathbb N}}
      \def\C{{\mathbb C}}

      \def\cA{\mathcal A}

      \def\cE{\mathcal E}
      \def\cS{\mathcal S}

      \def\cF{\mathcal F}

\usepackage{url}
\newcommand{\df}[1]{{\bf{#1}}{\index{#1}}}

\def\be{\bold{e}}

\title{A note on decomposable maps on operator systems}
\author[Balasubramanian]{Sriram Balasubramanian${}^*$}
\address{Department of Mathematics\\
 IIT Madras, Chennai - 600036, India.}
\email{bsriram@iitm.ac.in, bsriram80@yahoo.co.in}
\thanks{${}^*$Research supported by the grant MTR/2018/000113 from the Department of Science and Technology, India.}



\begin{document}

\maketitle
%\tableofcontents

\begin{abstract}
This article contains a characterization of operator systems $\cS$
 with the property that every positive map $\phi:\cS \rightarrow M_n$
 is decomposable, as well as an alternate and a more direct proof of
 a characterization of decomposable maps due to E. St\o rmer.
\end{abstract}

\begin{section}{Introduction}
 Let $H$ denote a Hilbert space over $\mathbb C$
  and $B(H)$ the C*-algebra of bounded operators on $H$.
  Let $\cA$ be a unital C*-algebra. Without loss of generality,
  we shall assume $\cA$ to be a C*-subalgebra of $B(H)$ for some
  Hilbert space $H$. An operator system $\cS \subseteq \cA$ is a unital self-adjoint
  subspace of $\cA.$ Letting $M_n=B(\C^n)$ to denote the C*-algebra
  of $n\times n$ complex matrices, a linear map $\phi:\cS\to M_n$
 is \df{positive} if $\phi(s)\succeq 0$ whenever $s$
 is a positive element of $\cS.$ Given a positive integer
 $k,$ let $\phi_k=\phi \otimes I_k: \cS\otimes M_k\to M_n \otimes M_k$
  denote the linear map determined by $\phi_k(s \otimes X)=\phi(s)\otimes X.$
 The map $\phi$ is \df{completely positive}, or \df{cp} for short,
 if each $\phi_k$ is positive; that is, if $S \in \cS\otimes M_k$ is positive
 as an element of the algebra $B(H) \otimes M_k = B(\oplus_1^k H)$ of $k\times k$
 matrices with entries from $B(H)$, then $\phi_k(S)$ is positive in
 $M_n \otimes M_k= B(\C^n \otimes C^k).$ 

%It is customary to write $\phi$ instead of $\phi_k$ when doing so causes no ambiguity.
Let $t$ denote a transpose on $M_n.$ A mapping $\phi:\cS\to M_n$ is
 \df{co-cp} if $t\circ \phi$ is cp. As is well known, the definition
 of co-cp is independent of the choice of transpose
 since any two transposes are unitarily equivalent.
 The linear map  $\phi$ is said to be \textbf{decomposable}
 if it is a sum of a cp map and a co-cp map.
 Maps that are positive, but not completely so, like the generic
 decomposable map, are of importance in quantum information theory as 
entanglement detecting maps.

 Let $\cS^+$ denote its positive elements of an operator space
 $\cS\subseteq \cA$.  Given a positive integer $n$,
  it is evident that $\cS^+\otimes M_n^+$, which is the cone
  generated by elementary tensors $s \otimes X$ where
 both $s$ and $X$ are positive, is a subset of $(\cS \otimes M_n)^+.$
   Operator systems $\cS$ with the property that every positive map
   $\phi:\cS\to M_n$ is completely positive
 are characterized as follows.  See for instance Theorem 6.6 in \cite{P}. 

\begin{proposition}
 \label{p:prealldecomp}
  Every positive map $\phi:\cS\to M_n$ is completely positive if
 and only if  $\cS^+ \otimes M_n^+$ is dense in $(\cS\otimes M_n)^+$.
\end{proposition}

  Using techniques from \cite{HOR} and \cite{W}, here we establish the
  analog of Proposition \ref{p:prealldecomp} for decomposable maps.
  Let $k \in \N$, $t$ denote a transpose on $M_k$ and
\begin{equation}
 \label{e:Jn}
   J_k(\cS) := \left \{S=\sum_{j=1}^k s_j \otimes x_j \,:\, S \succeq 0, \\ \sum_{j=1}^k  s_j\otimes t(x_j) \succeq 0 \right \}
      \subseteq \cS \otimes M_k.
\end{equation}

 \begin{theorem}
\label{thm:alldecomp}
Let $n \in \N$ and $\cS$ be an operator system in the unital C*-algebra $\cA$.
Every positive linear map $\psi: \cS \rightarrow M_n$ is decomposable if and only if
  $J_n(\cS) \subseteq \overline{\cS^+ \otimes M_n^+}$.
\end{theorem}

%It is well known that every positive linear map $\phi:M_p \rightarrow M_q$ is decomposable, whenever $p,q \in \N$ and $pq \le 6$.
%(please see \cite{St63}, \cite{Stbook}, \cite{HOR}  and  \cite{W}). By combining this fact with Theorem \ref{thm:alldecomp} and observing that $M_q^+ \otimes M_p^+$ is closed, one can conclude the following.
%\begin{corollary}
%\label{cor:posiffdecomp}
%If $p,q \in \N$ and $pq \le 6$, then $J_p(M_q) \subseteq M_q^+ \otimes M_p^+$.
%\end{corollary}

It is well known that every positive linear map $\phi:M_p \rightarrow M_q$ is decomposable, whenever $p,q \in \N$ and $pq \le 6$.
(please see \cite{St63}, \cite{Stbook}, \cite{HOR}  and  \cite{W}). By combining this fact with Theorem \ref{thm:alldecomp}, one can immediately conclude the following.
\begin{corollary}
\label{cor:posiffdecomp}
If $p,q \in \N$ and $pq \le 6$, then $J_p(M_q) \subseteq \overline{M_q^+ \otimes M_p^+}$.
\end{corollary}

%It is known that $J_p(M_q) \subseteq M_q^+ \otimes M_p^+$, whenever $pq \le 6$
%(please see \cite{St63}, \cite{Stbook}, \cite{HOR}, \cite{LMO} and  \cite{W}).
%Combining this with fact Theorem \ref{thm:alldecomp} yields the following well known result.
%\begin{corollary}
%\label{cor:posiffdecomp}
%If $p,q \in \N$ satisfy $pq \le 6$, then every positive linear map
%$\phi:M_p \rightarrow M_q$ is decomposable.
%\end{corollary}

The proof of Theorem~\ref{thm:alldecomp} is based upon a result of
 E. St{\o}rmer.  Let $L(\cS, M_n)$ denote the vector space of linear maps 
from $\cS$ to $M_n$. The \df{dual functional} $s_\phi:\cS\otimes M_n \rightarrow \C$, 
 associated to the linear map $\phi \in L(\cS, M_n)$, is the mapping defined by
 \begin{equation}
\label{eq:dualfunc}
s_\phi(s \otimes x) = \langle (\phi(s) \otimes x)\, \be,\be\rangle,
\end{equation}
where $\otimes$ denotes the Kronecker product, $\{e_1,\dots,e_n\}$ is the standard 
orthonormal basis of $\C^n$ and $\be = \sum_{j=1}^n e_j \otimes e_j \in \C^n \otimes \C^n$.  
It is customary to identify $M_n(\cS)$ with $\cS \otimes M_n$, via the mapping 
\begin{equation}
\label{eq:id}
M_n(\cS) \ni [x_{i,j}] \mapsto \sum_{i,j=1}^n x_{i,j} \otimes E_{i,j} \in \cS \otimes M_n,
\end{equation}
where $E_{i,j}=e_ie_j^*$ are the standard matrix units in $M_n$.  Under this identification, the dual functional $s_{\phi}:M_n(\cS) \rightarrow \C$ becomes  
\[
s_{\phi}([x_{i,j}]) = \langle [\phi(x_{i,j})] \be, \be \rangle,
\]
where $\be = e_1 \oplus e_2 \dots \oplus e_n \in \C^{n^2}$. 
It is also to be noted that the definition of $s_{\phi}:\cS \otimes M_n \rightarrow \C$ given above, coincides with that given in  \cite{St09} namely, 
\[
s_{\phi}(s \otimes x) = n(Trace(\phi(s)t(x))),
\] 
where $t(x)$ is the (standard) transpose of $x$. 

\begin{remark}
Suppose that $f:\cS \otimes M_n \to \C$ is linear, then with 
$\phi: \cS \to M_n$ denoting the linear map determined by
\begin{equation}
\label{eq:inverse}
  \langle \phi(s) e_k,e_j \rangle = f(s \otimes e_j e_k^*),
\end{equation}
one gets that $s_\phi = f.$  It follows from equations \eqref{eq:dualfunc} and \eqref{eq:inverse} that the mapping 
\[
L(\cS,M_n) \ni \phi \mapsto s_{\phi} \in L(\cS \otimes M_n, \C)
\]
is bijective.
\end{remark}

\begin{theorem}[\cite{St09}]
\label{t:decomposable}
Let  $\cS \subseteq \cA$ be an operator system, $\phi:\cS \rightarrow M_n$ be a linear map.
The map $\phi:\cS \rightarrow M_n$ is decomposable if and only if its associated dual
functional $s_{\phi}:M_n(\cS) \rightarrow \C$ satisfies $s_{\phi}(S) \ge 0$ whenever $S \in J_n(\cS)$.
\end{theorem}

A second contribution of this article is to give an
alternate and a more direct proof of Theorem \ref{t:decomposable}
by using the techniques developed in Chapter 6 of \cite{P}.
This approach also yields a simpler proof of a characterization
of a cp map $\phi:\cS \rightarrow M_n$ in terms of its associated
dual functional. Please see Theorem \ref{t:cp-criteria} in Section 3.
\end{section}

\begin{section}{Preliminaries}
This section contains some lemmas that will be used in the sequel.

Given an orthonormal basis $\cE$ of a Hilbert space $E$
the linear map  $t_{\cE}:B(E)\to B(E)$ uniquely determined by the property
\[
 \langle t_{\cE}(T) y, x\rangle = \langle Tx,y\rangle.
\]
for all $x,y \in \cE$ is positive and isometric and is
 the \textbf{transpose on $E$ associated to $\cE.$}

If $\cF$ is
 an orthonormal basis on a Hilbert space $F$, then
 $\cE\otimes \cF =\{e\otimes f: e\in \cE,\, f\in\cF\}$
 is an orthonormal basis for $E\otimes F$ and moreover,
\[
 t_{\cE\otimes \cF}  = t_{\cE}\otimes t_{\cF}.
\]
 In particular, $t_{\cE}\otimes t_{\cF}$ is a positive map.

 Given a unitary $U$ on $E$ the set $\cF=U\cE$ is also an orthonormal
 basis and an elementary computation shows, for $T\in B(E),$
\[
  t_{\cF}(T) = V t_{\cE}(T) V^*,
\]
 where $V=U t_{\cE}(U)^*$ is unitary. Thus any two transposes on $E$ are
 unitarily equivalent. As a consequence the notion of co-cp
 for a  linear map from an  operator space into $B(\C^n)$
 is independent of the choice of transpose (basis)
 on $B(\C^n)$ (of $\C^n$).

Recall the identification of $\cS \otimes M_n$ with $M_n(\cS)$ from \eqref{eq:id}. The following result (Lemma 1 in \cite{St08}) \& Lemma 6.5 in \cite{P}) explains the significance of the dual functional.  

 \begin{lemma}
 \label{l:sphi}
The linear map $\phi:\cS \rightarrow M_n$ is positive if and only if the linear
functional $s_{\phi}:M_n(\cS) \rightarrow \C$ takes positive values
on $\cS^+ \otimes M_n^+$.
\end{lemma}

The notion of  cp and co-cp maps easily extends to maps from
an operator system into $B(E)$ for a Hilbert space $E.$

\begin{lemma}
\label{l:etaJ}
  Suppose $E$ is a Hilbert space.
  If $\eta:\cS\to B(E)$ is co-cp, $m \in \N$ and $S \in J_m(\cS)$, then
 $(\eta \otimes I_m)(S)\succeq 0$.
\end{lemma}

\begin{proof}
Let $t$ denote a transpose on $B(E)$ and $t_m$ the standard transpose on $M_m.$
Suppose $S=\sum_{j=1}^m s_j \otimes x_j \in J_m(\cS).$ Thus $S$ and also
$S^\prime =\sum s_j \otimes t_m(x_j)$ are positive. Let $I_m$ denote
the identity operator on $M_m$. Since $\eta$ is co-cp,
 $S^\prime \succeq 0,$ and $t \otimes t_m$ is positive, it follows that 
\[
   0 \preceq (t \otimes t_m) (t \circ \eta\otimes I_m)(S^\prime)
  = (\eta\otimes t_m)(S^\prime) =  (\eta \otimes I_m)(S).
\]
\end{proof}
Recall the standard matrix units $E_{j,k} \in M_m$. The following is a key positivity property that will be utilized to prove our main results.

\begin{lemma}
 \label{l:pre4implies3}
   Suppose that $m \in \N$, $S=\sum_{j,k = 1}^m s_{j,k}\otimes E_{j,k} \in \cS\otimes M_m$,
   $y_1,\dots,y_m\in\mathbb C^n$ and $T= \sum_{j,k = 1}^m s_{j,k} \otimes y_j y_k^* \in
   \cS \otimes M_n.$
\begin{enumerate}[(i)]
\item    If $S \succeq 0,$ then $T\succeq 0.$
\item   If $S\in J_m(\cS),$ then $T\in J_n(\cS).$
\end{enumerate}
\end{lemma}

\begin{proof}
 (i) Let $1$ denote the unit element in $\cS$, $\{e_1, \dots, e_m\}$ denote the standard orthonormal basis for $\C^m$ and $Y = {1} \otimes \sum_{\alpha = 1}^m e_\alpha y_\alpha^*$. It follows that 
\[
 Y^*SY = \sum_{\alpha, \beta = 1}^m \sum_{j,k = 1}^m s_{j,k} \otimes y_\beta \, [e_\beta^* \, E_{j,k} \, e_\alpha]\,  y_\alpha^*
    = \sum_{\alpha, \beta = 1}^m  s_{\beta,\alpha} \otimes y_\beta y_\alpha^* = T,
\]
since $e_\beta^* E_{j,k} e_\alpha =1$ if $(\alpha,\beta)=(k,j)$ and $0$ otherwise.
Thus if $S \succeq 0$, then so is $T.$

 (ii) Let $z_j =\overline{y_j}$, the  entrywise complex conjugate.
 Suppose $S\in J_m(\cS).$  By definition of $J_m(\cS)$,
\[
 S^\prime = (I\otimes t_m)(S)=\sum_{j,k =1}^m s_{k,j} \otimes E_{j,k}\succeq 0,
\]
 where $I$ is the identity operator on $\cS$
  and $t_m$ is the transpose on $M_m.$
 From part (i) it follows that $T\succeq 0$ and
\[
 T^\prime = (I\otimes t_m)(T) =\sum_{j,k =1}^m s_{k,j} \otimes y_j y_k^*  = \sum_{j,k =1}^m s_{j,k} \otimes
  z_j z_k^* \succeq 0.
\]
 Hence $T\in J_n(\cS).$
\end{proof}

Recall the dual functional $s_{\phi}$ associated to $\phi$, from equation \eqref{eq:dualfunc}.
\begin{lemma}
 \label{l:sphibasic}
  Suppose $\cS\subset \cA$ is an operator system and $\phi:\cS\to M_n$
  is a linear map.    If $s\in \cS$ and $y,z\in \mathbb C^n$, then
 \[
 \langle \phi(s)\overline{z},\overline{y} \rangle = s_\phi( s\otimes yz^*).
\]
 Moreover, if $S=\sum_{j,k = 1}^m s_{j,k}\otimes E_{j,k}\in \cS\otimes M_m$
 and $w= w_1 \, \oplus w_2 \, \oplus \, \cdots \, \oplus \, w_m =
 \sum_{j=1}^m e_j\otimes w_j \in \C^m \otimes \C^n,$ then
\[
 \langle \phi_m(S) \overline{w}, \overline{w}\rangle
 = s_\phi \left (\sum_{j,k =1}^m s_{j,k}\otimes w_j w_k^* \right ).
\]
\end{lemma}
\begin{proof} Compute
\begin{align*}
 s_\phi(s\otimes yz^*) &= \left \langle (\phi(s) \otimes yz^*)
 \sum_{j =1}^n e_j\otimes e_j, \sum_{k =1}^n e_k\otimes e_k
  \right  \rangle \\
  & = \sum_{j,k} (e_k^* \phi(s) e_j)\,  (e_k^*y)(z^*e_j)\\
  & = \sum_{j,k} (e_j^* \phi(s) e_k)\, (z^*e_k)(e_j^*y)\\
  & = \langle \phi(s) \overline{z},\overline{y}\rangle.
\end{align*}

The second part can be obtained from the first part by linearity,
as follows.
\begin{align*}
s_\phi \left (\sum_{j,k =1}^m s_{j,k}\otimes w_j w_k^* \right )
 & = \sum_{j,k=1}^m s_\phi (s_{j,k}\otimes w_j w_k^*)\\
 & = \sum_{j,k=1}^m \langle \phi(s_{j,k}) \overline{w_k}, \overline{w_j} \rangle
 = \langle \phi_m(S) \overline{w}, \overline{w} \rangle.
\end{align*}
\end{proof}

\end{section}

\begin{section}{The Proofs}
This section contains our main results. We begin with the following theorem which is a
more elaborate version of our main Theorem \ref{t:decomposable} stated in Section 1.
 The elaboration is in the sense that it also integrates another characterization
 of decomposable maps on C*-algebras due to E. St\o rmer (\cite{St82}).

\begin{theorem}
\label{t:decomposable-criteria}
Let $\cA$ be a unital C*-algebra, $\cS \subseteq \cA$ be an operator system and $\phi:\cS \rightarrow M_n$ be a linear map.
The following statements are equivalent.
\begin{enumerate}[(i)]
\item \label{i:dc1} $\phi$ is decomposable.
\item \label{i:dc2}  $\phi_m(S) \succeq 0$ for all $m \in \N$ and $S \in J_m(\cS).$
\item \label{i:dc3} $\phi_n(S)\succeq 0,$ for all $S \in J_n(\cS).$
\item \label{i:dc4} The linear functional  $s_{\phi}:M_n(\cS) \rightarrow \C$
 is positive on $J_n(\cS)$.
\end{enumerate}
\end{theorem}

 To prove $\eqref{i:dc1} \Rightarrow \eqref{i:dc2}$,
  let $\phi = \psi +\eta,$ where $\psi $ is cp and $\eta$ is co-cp,
  $m \in \N$ and $S \in J_m(\cS)$.
  By Lemma \ref{l:etaJ}, $(\eta \otimes I_m )(S) \succeq 0$ and by the
  complete positivity of $\psi$, $(\psi \otimes I_m )(S) \succeq 0$.
  Hence $\phi_m(S) \succeq 0$.\\

 A proof of $\eqref{i:dc2} \Rightarrow \eqref{i:dc1} $ for the case $\cS = \cA$,
 can be found in \cite{St82}. With minor modifications, the proof can be
 made to work for any non-trivial operator system $\cS \subset \cA$.
 Hence, we omit the proof.\\

The implications $\eqref{i:dc2} \Rightarrow \eqref{i:dc3}$ and $\eqref{i:dc3} \Rightarrow \eqref{i:dc4}$ are immediate.\\

 Using Lemmas \ref{l:pre4implies3} and \ref{l:sphibasic} and
 techniques from Chapter 6 of \cite{P}, we give a streamlined proof
 of \eqref{i:dc4} implies \eqref{i:dc2} below. Let $e_1,\dots,e_m$ denote the standard orthonormal basis for $\C^m$ and $E_{j,k} = e_j e_k^*$, the resulting matrix units in $M_m$.

\begin{proof}[Proof of $(iv) \Rightarrow (ii)$]
Let $m \in \N$, $S=\sum_{j,k=1}^m s_{j,k} \otimes E_{j,k} \in J_m(\cS)$ and 
$w = w_1 \oplus \dots \oplus w_m  = \sum_{j=1}^m e_j \otimes w_j \in \C^m \otimes \C^n$. 
 From Lemma \ref{l:sphibasic},
\begin{equation}
 \langle \phi_m(S) \overline{w}, \overline{w}\rangle = s_\phi \left(\sum_{j,k=1}^m s_{j,k}\otimes w_j w_k^* \right)
 =s_{\phi}(T).
\end{equation}
where $T=\sum_{j,k=1}^m s_{j,k} \otimes w_jw_k^*$.  Since $S\in J_m(\cS),$ Lemma
 \ref{l:pre4implies3} implies that $T\in J_n(\cS)$. Thus, by hypothesis
 $s_\phi(T)\succeq 0$.
 Since $w \in \C^m \otimes \C^n$ is arbitrary, it follows that $\phi_m(S)$ is positive and the
 proof is complete.
\end{proof}

The following theorem can be found in Chapter 6 of \cite{P}. The proof given there
uses the fact that every positive matrix in $M_k(\cA)$ is a finite sum of matrices of the form $[a_i^*a_j]$, where $a_1, \dots, a_k \in \cA$. It is observed that, one can obtain a proof without using this property, by using Lemma \ref{l:pre4implies3} instead, as indicated below.
\begin{theorem}
\label{t:cp-criteria}
Let $\cA$ be a unital C*-algebra, $\cS \subseteq \cA$ be an operator system and $\phi:\cS \rightarrow M_n$ be a linear map.
The following statements are equivalent.
\begin{enumerate}[(i)]
\item \label{i:cp1} $\phi$ is cp.
\item \label{i:cp2} The linear functional $s_{\phi}:M_n(\cS) \rightarrow \C$
is positive on $(\cS\otimes M_n)^+.$ %$s_{\phi}([x_{ij}]) \ge 0$ whenever $[x_{ij}] \in M_n(\cS)^+$.
\end{enumerate}
\end{theorem}
\begin{proof}
That  \eqref{i:cp1} implies \eqref{i:cp2}  is immediate from the complete positivity of $\phi$ and the definition of $s_{\phi}$.
 To prove \eqref{i:cp2} implies \eqref{i:cp1},  let $m \in \N$ and
 $S=\sum_{j,k = 1}^m s_{j,k}\otimes E_{j,k} \in (\cS\otimes M_m)^+$ be given.
 Given $w =\sum_{j,k=1}^m e_j\otimes w_j \in \C^m \otimes \C^n$, it follows from
 part (i) of Lemma~\ref{l:pre4implies3}
  that $T=\sum s_{j,k} \otimes  w_j w_k^*  \in (\cS\otimes M_n)^+.$
  Hence, using Lemma \ref{l:sphibasic},
\[
 \langle \phi_m(S)\overline{w},\overline{w}\rangle = s_\phi(T) \succeq 0.
\]
 Thus $\phi_m(S)$ is positive and the result follows.
\end{proof}

%The proof of Theorem~\ref{thm:alldecomp} below uses the following
%observation. If $\cS$ is an operator system and  {\color{blue}$f:\cS\otimes M_n \to \C$ 
%is linear,} then, with $\phi:\cS\to M_n$ denoting the linear map determined by
%\[
%  \langle \phi(s)e_k,e_j \rangle = f(s\otimes e_j e_k^*),
%\]
%we have $s_\phi = f.$
%p = \sum_{i.j=1}^n p_{i,j} \otimes E_{i,j}

\begin{proof}[Proof of Theorem~\ref{thm:alldecomp}.]
$(i) \Rightarrow (ii)$: Suppose not. Choose $p \in J_n(\cS)$ such that $p \not \in \overline{\cS^+ \otimes M_n^+}$. Let $A = \{p\}$ and $B = \overline{\cS^+ \otimes M_n^+}$. Observe that $A$ and $B$ satisfy the hypotheses of the
Hahn Banach separation theorem \cite[Theorem 3.4]{R}. It follows that there exists a continuous linear functional $\Lambda:\cS \otimes M_n \rightarrow \C$ and $\gamma_1, \gamma_2 \in \R$ such that
\[
Re(\Lambda(p)) < \gamma_1 < \gamma_2 <  Re(\Lambda(x))
\]
for all $x \in B$. Since $0 \in B$, it must be the case that $\gamma_2 \le 0$. Suppose that $Re(\Lambda(x_0)) < 0$ for some $x_0 \in B$. Since $B$ is a cone, $nx_0 \in B$ for all $n \in \N$. The above equation implies that $Re(\Lambda(nx_0)) = n Re(\Lambda(x_0)) > \gamma_2$ for all $n \in \N$. This is impossible, since $\gamma_2 \le 0$. Thus,
\[
Re(\Lambda(p)) < \gamma_1 < \gamma_2 \le 0 \le Re(\Lambda(x)),
\]
for all $x \in B$. Define $f:M_n(\cS) \rightarrow \C$ by $f(x) = \frac{1}{2}\left(\Lambda(x) + \overline{\Lambda(x^*)} \right)$. Observe that $f$ is a continuous linear functional which satisfies
\begin{equation}
\label{eq:natureofs}
f(p) < 0 \text{ and } f(x) \ge 0
 \end{equation}
 for all $x \in B$.  By equation \eqref{eq:inverse}, there exists $\phi: \cS\rightarrow M_n$ such that $f = s_{\phi}$.  Since $f$ is positive on $B$, by  Lemma \ref{l:sphi}, it follows that $\phi: \cS \rightarrow M_n$ is positive. Since $p \in J_n(\cS)$ and $f(p) = s_{\phi}(p) < 0$, it follows from Theorem \ref{t:decomposable-criteria} that $\phi:\cS \rightarrow M_n$ is not decomposable, a contradiction.\\

$(ii) \Rightarrow (i)$: Let $\psi:\cS \rightarrow M_n$ be a positive map. It follows from Lemma \ref{l:sphi} that, $s_\psi$ takes positive values on $\cS^+ \otimes M_n^+$, and hence also on $\overline{\cS^+ \otimes M_n^+}$. Since $J_n(\cS) \subseteq \overline{\cS^+ \otimes M_n^+}$, it follows that $s_{\psi}$ takes positive values on $J_n(\cS)$. An application of Theorem \ref{t:decomposable-criteria} yields the decomposability of $\psi$, and the proof is complete.
\end{proof}
Following \cite{St82}, we end with an application of Theorem \ref{t:decomposable-criteria}.
\begin{example}\rm
Consider the map $\phi:M_3 \rightarrow M_3$ defined by
\begin{equation}
\label{eq:choimap}
\phi \begin{pmatrix} x_{11} & x_{12} & x_{13}\\ x_{21} & x_{22} & x_{23} \\ x_{31} & x_{32} & x_{33} \end{pmatrix}  = \begin{pmatrix} x_{11} & -x_{12} & -x_{13}\\ -x_{21} & x_{22} & -x_{23} \\ -x_{31} & -x_{32} & x_{33} \end{pmatrix} + \mu \begin{pmatrix} x_{33} & 0 & 0\\ 0 & x_{11} & 0 \\ 0 & 0 & x_{22} \end{pmatrix},
\end{equation}
where $\mu \ge 1$. It was shown by M.D. Choi that the above map is a positive map but not decomposable (See \cite{C1} and \cite{C2}). Consider the matrix
\begin{equation}
\label{eq:ppt}
A(a) := \left( \begin{array}{ccc|ccc|ccc}
1 & 0 & 0 & 0 & 1 & 0 & 0 & 0 & 1 \\
0 & 1/a & 0 & 0 & 0 & 0 & 0 & 0 & 0 \\
0 & 0 & a & 0 & 0 & 0 & 0 & 0 & 0 \\
\hline
0 & 0 & 0 & a & 0 & 0 & 0 & 0 & 0 \\
1 & 0 & 0 & 0 & 1 & 0 & 0 & 0 & 1 \\
0 & 0 & 0 & 0 & 0 & 1/a & 0 & 0 & 0 \\
\hline
0 & 0 & 0 & 0 & 0 & 0 & 1/a & 0 & 0\\
0 & 0 & 0 & 0 & 0 & 0 & 0 & a & 0\\
1 & 0 & 0 & 0 & 1 & 0 & 0 & 0 & 1
\end{array} \right).
\end{equation}
We note that the matrix $A(a)$ is a minor refinement of the matrix that appears in page 403 of \cite{St82} and that $A(a)$ belongs to $J_3(M_3)$, if $a > 0 $ (Ex. 5(a) on Page 32 of \cite{J}). Also observe that $s_{\phi}(A(a)) = (a \mu - 1)$. Since $\mu \ge 1$, if one chooses $0 < a < \frac{1}{\mu}$, then it follows easily from Theorem \ref{t:decomposable-criteria} that $\phi:M_3 \rightarrow M_3$ is not decomposable. Since $\phi$ is a positive map, using Lemma \ref{l:sphi}, one can also conclude that the matrix $A(a)$ does not belong to $\overline{M_3^{+} \otimes M_3^{+}}$, whenever $0 < a < \frac{1}{\mu}$. 
\end{example}
\end{section}

\textbf{Acknowledgements:} The author would like to thank Prof. Scott McCullough for many helpful suggestions and Prof. Erling St\o rmer for useful comments during the initial stages of this work.

% and the Department of Science and Technology (DST), India, for financial support in the form of the MATRICS research grant MTR/2018/000113.

\end{document}